\documentclass[12pt]{article}

\usepackage{subfigure}
\usepackage{color}
\usepackage[numbers,sort&compress]{natbib}
\usepackage[english]{babel}
\usepackage{amsmath,amsthm}
\usepackage{amsfonts}
\usepackage{float}
\usepackage{indentfirst}
\usepackage[hang,flushmargin]{footmisc}
\allowdisplaybreaks[4]

\usepackage{graphicx}
\newtheorem{thm}{Theorem}
\newtheorem{cla}{Claim}
\newtheorem{lem}{Lemma}

\newtheorem{ob}{Observation}
\newtheorem{proof of Thm}{proof of Thm}

\begin{document}

\begin{sloppypar}

\baselineskip 0.65cm

\centerline{\large\large The Outerplanar Tur\'{a}n Number of Double Stars}\vskip .1in

\renewcommand{\thefootnote}{}

\footnotetext{
Corresponding author. E-mail: yxlan@hebut.edu.cn.}

\centerline{\small Chaofan Zhang, Yongxin Lan$^{\ast}$, Changqing Xu
 } \vskip.12in

\centerline{\small School of Science, Hebei University of Technology, Tianjin 300401, China}
\vskip.25in

\begin{abstract}
Let $H$ be a nonempty graph. A graph is $H$-free if it does not contain any copy of $H$ as a subgraph. The outerplanar Tur\'{a}n number of $H$, denoted by $ex_{_\mathcal{OP}}(n,H)$, is the maximum number of edges among all $H$-free outerplanar graphs on $n$ vertices. A double star $S_{p,q}$ is the graph obtained from an edge by joining its two endpoints with $p$ and $q$ isolated vertices respectively, where $q \ge p\ge 1$. In this paper, we determine the exact values of $ex_{_\mathcal{OP}}(n,S_{p,q})$ for all $q\ge p\ge 2$, with the sole exception of $p=2$ and $q=3$; for the latter, we establish a lower bound.

\bigskip

\noindent\textbf{Keywords:} Outerplanar graph; Tur\'{a}n number; Double star
\end{abstract}

\section{Introduction}

In this paper, we consider only finite, undirected and simple graphs. Let $G$ be a graph with vertex set $V(G)$ and edge set $E(G)$. We use $\delta(G)$ and $\Delta(G)$ to denote the minimum and maximum degree of $G$. For a subset $S$~(or $E$) of $V(G)$~(or $E(G)$), the induced subgraph on $S$~(or $E$) is denoted by $G[S]$~(or $G[E]$), and $G \setminus S$ denotes the induced subgraph on $V(G) \setminus S$. For a vertex $v$ in $V(G)$, $N_G(v)$ denotes the set of vertices adjacent to $v$ in $G$, and the degree of $v$ in $G$, written $d_{G}(v)$, is defined as $|N_{G}(v)|$. For a subgraph $H$ of $G$ and a vertex $v$ in $V(H)$, let $N(v,G\setminus V(H))=N_{G}(v)\setminus N_{H}(v)$ and $d(v,G\setminus V(H))=|N(v,G\setminus V(H))|$. Let $C_{k}$ and $P_{k}$ denote the cycle and path on $k$ vertices, respectively. A double star $S_{p,q}$ is the graph obtained from an edge by joining its two endpoints with $p$ and $q$ isolated vertices respectively, where $q \ge p\ge 1$.

A block is a maximal nonseparable subgraph. An endblock of a connected graph $G$ is a block that contains exactly one cut vertex (except when the graph itself is a block). If blocks $B_1$ and $B_2$ share exactly one common vertex, then $B_1$ and $B_2$ are adjacent. Given a block $B$ of $G$, $v\in V(G)$ is said to be adjacent to $B$ if $v$ is a cut vertex of $G$ and $v\in V(B)$. A planar graph is outerplanar if it has a planar embedding such that all vertices lie on the boundary of its outer face, such a planar embedding is called an outerplane graph. An outerplanar graph $G$ is maximal if the graph obtained from $G$ by joining any two non-adjacent vertices of $G$ is not outerplanar. Let $\mathcal{M}_k$ denote the family of maximal outerplanar graphs of order $k$~($k\ge 1$). Since $|\mathcal{M}_k| = 1$ when $1 \le k \le 5$, we simply denote $\mathcal{M}_k$ as $M_k$ in the following. For all positive integer $k$, let $[k]=\{1,2,\cdots,k\}$. Terminology and notation not introduced in this paper can be referred to~\cite{BM}.

Let $H$ be a nonempty graph. A graph is $H$-free if it does not contain any copy of $H$ as a subgraph. A central problem in extremal graph theory is the Tur\'{a}n problem: determining the maximum number of edges in an $n$-vertex $H$-free graph and characterising all extremal graphs. This problem gives rise to the concept of the Tur\'{a}n number $ex(n, H) $, which denotes the maximum number of edges in an $n$-vertex $H$-free graph.

In 1984, Erd\H{o}s~\cite{E84} initiated research on the Tur\'{a}n number of $H$ in an $n$-dimensional hypercube $Q_n$, denoted by $ex(Q_n, H)$, by studying the case where $H$ is an even cycle; he conjectured that $ex(Q_n, C_{4})=(\frac{1}{2} + o(1))e(Q_n)$. Since then, researchers such as Baber~\cite{B12}, Chung~\cite{Ch92}, Conder~\cite{C93} and Grebennikov and Marciano~\cite{G25} have made contributions to the study of $ex(Q_n, C_{2k})$ for $k\ge 2$, but its exact value remains unknown. Very recently, when $H$ is a double star, Liu and Xu~\cite{LX} determined $ex(Q_n ,S_{n-1, n-1})=(4n - 3)2^{n-3}$ for all $n \geq 3$.

Tur\'{a}n-type problems for double stars have also been extensively studied in planar graphs. In 2016, Dowden~\cite{Dow16} initiated the study of planar Tur\'{a}n-type problems. The planar Tur\'{a}n number $ex_{_\mathcal{P}}(n, H)$ is defined as the maximum number of edges in an $n$-vertex $H$-free planar graph. In 2019, Lan, Shi and Song~\cite{LSS19a} proved that if $H$ is $K_{4}$-free and $\Delta(H)\ge 7$, then $ex_{_\mathcal{P}}(n, H) = 3n-6$ for all $n \ge |V(H)|$. So for all $q\ge 6$, $ex_{_\mathcal{P}}(n, S_{p,q})=3n-6$ as $S_{p,q}$ is $K_{4}$-free. For all $p\ge 4$, $ex_{_\mathcal{P}}(n, S_{p,q})=3n-6$ since the double wheel graph $2K_{1}+C_{n-2}$ is $S_{p,q}$-free~\cite{XZLbal}. Only the cases when $p \le 3$ and $q \le 5$ remain to be further considered. In 2022, Ghosh, Gy\H{o}ri, Paulos and Xiao~\cite{GGPX22} pioneered the systematic study of the planar Tur\'{a}n number of double stars $S_{p,q}$ for $q\ge p\ge 2$ and achieved groundbreaking results: they determined its exact value when $p+q \le 5$; they derived its lower and upper bounds when $6\le p+q \le 7$, where the upper bounds when $p+q = 6$ were improved by Xu and Shao~\cite{XuS24} and Xu, Zhou, Li and Yan~\cite{XZLbal}; the upper bounds when $p+q = 7$ were improved by Xu, Shao and Zhou~\cite{X2l} and Xu, Zhang and Shao~\cite{X25a}. In 2025, Liu and Xu~\cite{X25b} extended this line of research by establishing an upper bound of $ex_{_\mathcal{P}}(n, S_{3,5})$. The best current results are presented in Theorem~\ref{thzj}.

\begin{thm}
{\rm{\label{thzj}}}
Let $n$ be a positive integer.

{\noindent}(a)~\textup{\cite{GGPX22}}~$ex_{_\mathcal{P}}(n,S_{2,2})= 2n-4$ for all $n\ge 16$.

{\noindent}(b)~\textup{\cite{GGPX22}}~$ex_{_\mathcal{P}}(n,S_{2,3})= 2n$ for all $n\ge 6$.

{\noindent}(c)~\textup{\cite{XuS24}}~$ex_{_\mathcal{P}}(n, S_{2,4})\le \frac{31n}{14}$ for all $n\ge 1$, equality holds when $n\equiv 0~(\emph{mod}~14)$.

{\noindent}(d)~\textup{\cite{XZLbal}}~$ex_{_\mathcal{P}}(n, S_{3,3})=\lfloor\frac{5n}{2}\rfloor-5$ for all $n\ge 10$.

{\noindent}(e)~\textup{\cite{GGPX22,X2l}}~$\frac{5n}{2}\le ex_{_\mathcal{P}}(n, S_{2,5})\le \frac{19n-18}{7}$ for all $n\ge 1$.

{\noindent}(f)~\textup{\cite{X25a}}~$ex_{_\mathcal{P}}(n, S_{3,4})\le \frac{5n}{2}$ for all $n\ge 1$, equality holds when $n\equiv 0~(\emph{mod}~12)$.

{\noindent}(g)~\textup{\cite{X25b}}~$ex_{_\mathcal{P}}(n, S_{3,5})\le \frac{23n}{8}-\frac{9}{2}$ for all $n\ge 2$, equality holds when $n = 12$.
\end{thm}

Let $ex_{_\mathcal{OP}}(n,H)$~($ex_{_\mathcal{OP}}^{c}(n,H)$) denote the maximum number of edges among all (connected) $H$-free outerplanar graphs on $n$ vertices. In this paper, we focus on the outerplanar Tur\'{a}n number of double stars. We determine the exact values of $ex_{_\mathcal{OP}}(n,S_{p,q})$ for all $q\ge p\ge 2$, with the sole exception of the special case $p=2$ and $q=3$.

To determine $ex_{_\mathcal{OP}}(n, S_{2,2})$, we first prove $ex_{_\mathcal{OP}}^c(n, S_{2,2})$ in Theorem~\ref{jg1} as a foundation, then show in Theorem~\ref{jg2} that $ex_{_\mathcal{OP}}(n, S_{2,2})$ matches it for all $n \ge 6$ except $n = 10$.
\smallskip

\begin{thm}
{\rm{\label{jg1}}}
Let $n$ be a positive integer. We have $ex_{_\mathcal{OP}}^{c}(n,S_{2,2})= \lfloor\frac{3(n-1)}{2}\rfloor$ for all $n\ge 6$.
\end{thm}

\begin{thm}
{\rm{\label{jg2}}}
Let $n$ be a positive integer with $n\ge 6$. We have
$$ex_{_\mathcal{OP}}(n,S_{2,2})={\left\{\begin{array}{ll}
  \lfloor\frac{3(n-1)}{2}\rfloor,& if~n\ge 6~and~n\neq 10; \\
  14,& if~n=10.
\end{array}\right.}$$
\end{thm}

\begin{thm}
{\rm{\label{jg3}}}
Let $p,q,n$ be positive integers with $n\ge p+q+2$ and $q\ge p\ge 2$. If $p\ge 3$ or $q\ge 4$, then
$$ex_{_\mathcal{OP}}(n,S_{p,q})= 2n-3.$$
\end{thm}

We establish a lower bound for $ex_{_\mathcal{OP}}(n,S_{2,3})$ in Theorem~\ref{jg4} and propose a conjecture regarding its exact value. Furthermore, we conjecture that $ex_{_\mathcal{OP}}(n,S_{2,3})=ex_{_\mathcal{OP}}^{c}(n,S_{2,3})$ for all $n\ge 1$.

\begin{thm}
{\rm{\label{jg4}}}
Let $n$ be a positive integer with $n\ge 7$. We have
$$ex_{_\mathcal{OP}}(n,S_{2,3})\ge{\left\{\begin{array}{ll}
  \frac{5n-3}{3},& if~n\equiv 0~(\emph{mod}~6); \\
  \frac{5n-5}{3},& if~n\equiv 1~(\emph{mod}~6); \\
  \frac{5n+i-9}{3},& if~n\equiv i~(\emph{mod}~6),~where~i\in\{2,3,4,5\}.
\end{array}\right.}$$
\end{thm}

\noindent{\bf Conjecture~6.}~\textit{Let $n$ be a positive integer with $n\ge 1$. We have}
$$ex_{_\mathcal{OP}}(n,S_{2,3})=ex_{_\mathcal{OP}}^{c}(n,S_{2,3})={\left\{\begin{array}{ll}
  \frac{5n-3}{3},& if~n\equiv 0~(\textup{mod}~6); \\
  \frac{5n-5}{3},& if~n\equiv 1~(\textup{mod}~6); \\
  \frac{5n+i-9}{3},& if~n\equiv i~(\textup{mod}~6),~where~i\in\{2,3,4,5\}.
\end{array}\right.}$$

\section{Preliminaries}

For the convenience of subsequent proofs, we denote the double star $S_{p,q}$ by $S_{p,q}(x,y;X,Y)$, where $x$ and $y$ are central vertices of $S_{p,q}$, $X=N_{S_{p,q}}(x)\setminus \{y\}$ and $Y=N_{S_{p,q}}(y)\setminus \{x\}$ with $X\cap Y=\emptyset$~(see Figure $\ref{gds}$). A $(a,b)$-degree edge is an edge whose endpoints are of degree $a$ and $b$, respectively. Obviously, $d_{S_{p,q}}(x)=p+1$ and $d_{S_{p,q}}(y)=q+1$, so $S_{p,q}$ contains a $(p+1,q+1)$-degree edge $xy$.

\begin{figure}[htp]
	\centering
	\includegraphics[width=6cm]{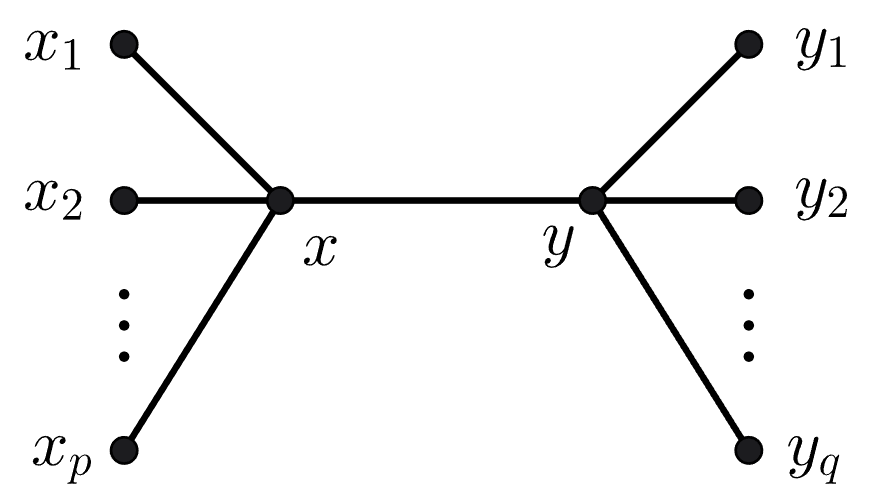}
	\caption{ The graph $S_{p,q}(x,y;X,Y)$, where $X=\{x_{1},x_{2},\ldots,x_{p}\}$ and $Y=\{y_{1},y_{2},\ldots,y_{q}\}$.}
	\label{gds}
\end{figure}

To prove our main theorems, we need the following lemmas concerning outerplanar graphs.

\begin{lem}$ {\emph{\cite{BM}}}$
{\rm{\label{lem1}}}
Let $G$ be an outerplanar graph with at least $3$ vertices. If $G$ is $2$-connected, then $G$ has a Hamilton cycle.
\end{lem}

\begin{lem}$ {\emph{\cite{BM}}}$
{\rm{\label{lem2}}}
Let $n$ be a positive integer. If $G$ is a maximal outerplanar graph with $n$ vertices, then $|E(G)|=\emph{max}\{0,2n-3\}$.
\end{lem}

\begin{lem}$ {\emph{\cite{WDB}}}$
{\rm{\label{lem3}}}
Let $r$ be a positive integer. A connected graph with blocks $G_{1},G_{2}, \ldots, G_{r}$ has $ \sum_{i=1}^{r} |V(G_{i})| - r + 1 $ vertices.
\end{lem}

\begin{lem}
{\rm{\label{lem12}}}
Let $n$ be an integer with $1\le n\le p+q+1$. We have $ex_{_\mathcal{OP}}(n,S_{p,q})=ex_{_\mathcal{OP}}^{c}(n,S_{p,q}) = \emph{max}\{0,2n-3\}$.
\end{lem}

\begin{proof}[\textbf{Proof}]
Since $|S_{p,q}|=p+q+2>n$, we get that any maximal outerplanar graph with $n$ vertices is $S_{p,q}$-free. Therefore, Lemma~\ref{lem12} is trivially true by Lemma~\ref{lem2}.
\end{proof}

\begin{lem}
{\rm{\label{lem11}}}
Let $G$ be an $S_{2,2}$-free outerplanar graph with at least 6 vertices. For any $x\in V(G)$ and $y\in V(G)$, if $xy\in E(G)$, $d_{G}(x)\ge 3$ and $d_{G}(y)\ge 3$, then $|N_{G}(x)\cap N_{G}(y)|\ge 1$.
\end{lem}

\begin{proof}[\textbf{Proof}]
By contradiction. Suppose that $|N_{G}(x)\cap N_{G}(y)|=0$, $\{x_1, x_2\}\subseteq N_G(x)\setminus\{y\}$ and $\{y_1, y_2\}\subseteq N_G(y)\setminus\{x\}$, where $\{x_1, x_2, y_1, y_2\} \subseteq V(G)$. Thus, $G$ contains an $S_{2,2}(x,y;\{x_{1},x_{2}\},\{y_{1},y_{2}\})$, contradicting $G$ is $S_{2,2}$-free. Therefore, $|N_{G}(x)\cap N_{G}(y)|\ge 1$.
\end{proof}

\begin{lem}
{\rm{\label{m5}}}
Let $G$ be an $S_{2,2}$-free connected outerplanar graph. If $|V(G)|\ge 6$, then $G$ is $M_{5}$-free.
\end{lem}

\begin{proof}[\textbf{Proof}]
Suppose not. Let $H_{1}\subseteq G$, $H_{1}\cong M_{5}$ and $V(H_{1})=\{u_{1},u_{2},\ldots,u_{5}\}$~(see Figure~\ref{figurem5}). Since $|V(G)| \ge 6$ and $G$ is connected, there must exist a vertex $u$ in $V(G) \setminus V(H_{1})$ such that $u$ is adjacent to some vertex in $H_{1}$. If $uu_{1}\in E(G)$, then $G$ contains $S_{2,2}(u_{1},u_{3};\{u,u_{5}\},\{u_{2},u_{4}\})$, a contradiction. If $uu_{i}\in E(G)$ for $i\in \{2,3\}$, then $G$ contains $S_{2,2}(u_{i},u_{1};\{u,u_{5-i}\},\{u_{4},u_{5}\})$, a contradiction. By the symmetry of $H_{1}$, all other cases (e.g., $u$ is adjacent to $u_4$ or $u_5$) lead to the same contradiction. Thus, $G$ is $M_{5}$-free, as desired.
\end{proof}
\begin{figure}[htbp]
    \centering
    \subfigure[~The graph $H_{1}$.]{
        \includegraphics[width=4.2cm]{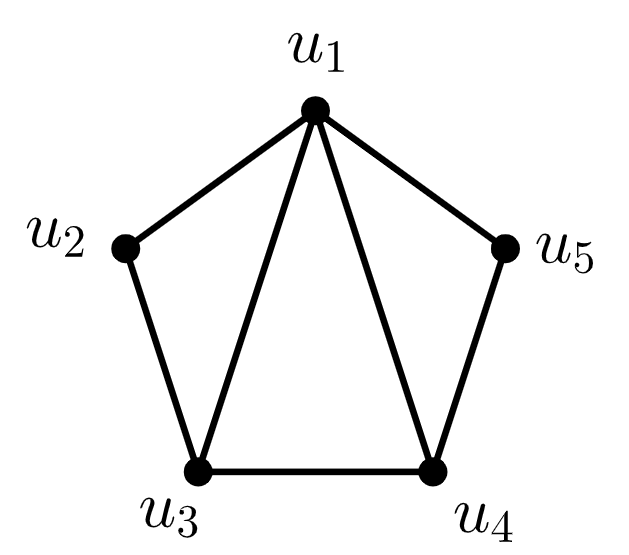}
        \label{figurem5}
    }
    \hspace{0.05\textwidth} 
    \subfigure[~The graph $H_{2}$.]{
        \includegraphics[width=4.2cm]{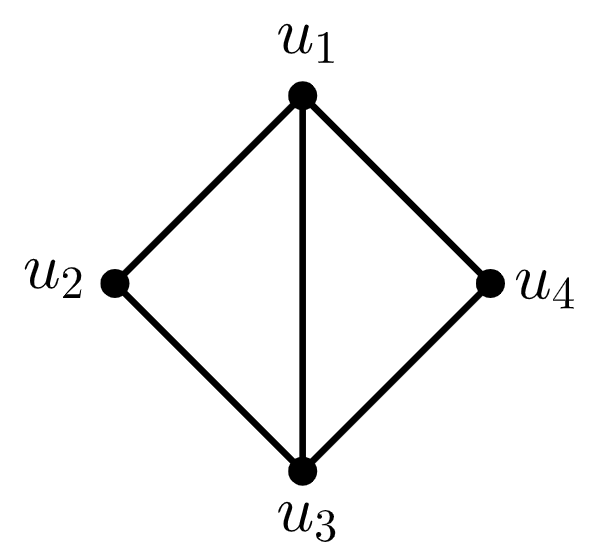}
        \label{figurem4}
    }
    \caption{~The graphs $H_{1}$ and $H_{2}$.}
\end{figure}

\begin{lem}
{\rm{\label{m4}}}
Let $G$ be an $S_{2,2}$-free connected outerplanar graph with $|V(G)|\ge 6$. If a block $B$ of $G$ is adjacent to a block that is isophormic to $M_{4}$, then we have

{\noindent}\emph{(1)}~$B\cong M_{2}$;

{\noindent}\emph{(2)}~$B$ is adjacent to at most two blocks. Further, if $B$ is adjacent to two blocks, then the another block adjacent to $B$ is $M_2$.
\end{lem}

\begin{proof}[\textbf{Proof}]

Let $H_{2}\subseteq G$, $H_{2}\cong M_{4}$, $V(H_{2})=\{u_{1},u_{2},u_{3},u_{4}\}$~(see Figure~\ref{figurem4}) and $H_{2}$ is adjacent to $B$. We first prove the following claim.
\begin{cla}
{\rm{\label{cla0}}}
For any $v\in V(H_{2})$, we have $d(v,G\setminus V(H_{2}))\le 1$.
\end{cla}
\textit{Proof}.~By contradiction. Suppose that there exists some $v$ in $V(H_{2})$ such that $d(v,G\setminus V(H_{2}))\ge 2$. Let $\{v_{1}, v_{2}\} \subseteq N(v,G\setminus V(H_{2}))$. If $v=u_{i}$~($i\in \{1,2\}$), then there exists an $S_{2,2}(u_{i},u_{3};\{v_{1},v_{2}\},\{u_{3-i},u_{4}\})$ in $G$, a contradiction. By the symmetry of $H_{2}$, all choices of $v$ lead to the similar contradiction. Thus, Claim~\ref{cla0} holds.
$\hfill \square$
\bigskip

Let $v_c$ be the unique cut vertex shared by $H_{2}$ and $B$, and let $v_{c}'\in N_{B}(v_{c})$. By Claim~\ref{cla0}, $d(v_{c},G\setminus V(H_{2}))\le 1$. We present the next claim.
\begin{cla}
{\rm{\label{cla00}}}
$d_{G}(v_{c}')\le 2$.
\end{cla}
\textit{Proof}.~By contradiction. Suppose that $d_{G}(v_{c}')\ge 3$. Combining this with $\delta(H_{2})=2$, we obtain $d_{G}(v_{c})\ge 3$. By Lemma~\ref{lem11}, $|N_{G}(v_{c})\cap N_{G}(v_{c}')|\ge 1$. Since $v_{c}$ is a cut vertex, we have $d(v_{c},G\setminus V(H_{2}))\ge 2$, contradicting $d(v_{c},G\setminus V(H_{2}))\le 1$. Thus, Claim~\ref{cla00} holds.
$\hfill \square$
\bigskip

Since $V(B)\cap V(H_{2})=\{v_c\}$ and $d(v_{c},G\setminus V(H_{2}))\le 1$, we have $|V(B)| \le 2$. So $B\cong M_{2}$, the conclusion (1) holds. By Claim~\ref{cla00}, $v_{c}'$ can be adjacent to at most one block $B'$ other than $B$. Further, if $B'$ exists, then $B'\cong M_{2}$. Notice that $v_c$ is adjacent to exactly one block other than $B$ and $B$ only contains the two vertices $v_{c}$ and $v_{c}'$. We can easily conclude that conclusion (2) holds.
\end{proof}

\begin{lem}
{\rm{\label{lemjg1}}}
Let $G$ be an $S_{2,2}$-free outerplanar graph with $n$ vertices. If $n\ge 6$ and $G$ is $2$-connected, then $|E(G)|\le \lfloor\frac{5n}{4}\rfloor$.
\end{lem}
\begin{proof}[\textbf{Proof}]
By Lemma $\ref{lem1}$, $G$ contains a Hamilton cycle, denoted by $C=v_{1}v_{2}\ldots v_{n}v_{1}$. We first show that $\Delta(G)\le 3$. By contradiction. Suppose that there exists a vertex of degree at least $4$. Without loss of generality, assume that $d_{G}(v_{1})\ge 4$ and $\{v_{i},v_{j}\}\subseteq N_{G}(v_{1})\setminus \{v_{2},v_{n}\}$, where $3\le i<j\le n-1$. If $i>3$, then there exists an $S_{2,2}(v_{1},v_{i};\{v_{2},v_{n}\},\{v_{i-1},v_{i+1}\})$ in $G$ as $i-1>2$ and $i+1<n$, contradicting $G$ is $S_{2,2}$-free. If $j<n-1$, then there exists an $S_{2,2}(v_{1},v_{j};\{v_{2},v_{n}\},\{v_{j-1},v_{j+1}\})$ in $G$ as $j-1>2$ and $j+1<n$, a contradiction. If $i=3$ and $j=n-1$, then there exists an $S_{2,2}(v_{1},v_{3};\{v_{n-1},v_{n}\},\{v_{2},v_{4}\})$ in $G$ as $n\ge 6$, a contradiction. So $\Delta(G)\le 3$.

Let $V_{i}(G)=\{v|d_{G}(v)=i,v\in V(G)\}$, where $i\in \{2,3\}$. We next show that $|V_{3}(G)|\le \lfloor\frac{n}{2}\rfloor$. Notice that for each $k\in [n]$, $\{v_{k-1},v_{k+1}\}\subseteq N_{G}(v_{k})$, where $v_{0}=v_{n}$ and $v_{n+1}=v_{1}$. Assume that $v_{\ell}\in V_{3}(G)$, where $\ell\in [n]$. We are going to prove that $v_{\ell+1}\notin V_{3}(G)$. Suppose, for contradiction, that $v_{\ell+1}\in V_{3}(G)$. By $v_{\ell}v_{\ell+1}\in E(G)$ and Lemma $\ref{lem11}$, $|N_{G}(v_{\ell})\cap N_{G}(v_{\ell+1})|\ge 1$. Let $v_{r}\in N_{G}(v_{\ell})\cap N_{G}(v_{\ell+1})$, where $v_{n+2}=v_{2}$. Notice that $v_{\ell}v_{\ell+2}\notin E(G)$ and $v_{\ell-1}v_{\ell+1}\notin E(G)$ as $G$ is outerplanar and $\{v_{\ell},v_{\ell+1}\}\subseteq V_{3}(G)$. Hence $r \in [n]\setminus \{\ell-1,\ell,\ell+1,\ell+2\}$. So $\{v_{r-1},v_{r+1},v_{\ell},v_{\ell+1}\}\subseteq N_{G}(v_{r})$ and thus $d_{G}(v_{r})\ge 4$, contradicting $\Delta(G)\le 3$. Thus, $v_{\ell+1}\notin V_{3}(G)$. Therefore, $|V_{3}(G)|\le \lfloor\frac{n}{2}\rfloor$.

By the Handshaking Lemma, $2|E(G)|=\underset{v\in V(G)}{\sum}d_{G}(v)=2|V_{2}(G)|+3|V_{3}(G)|=2(n-|V_{3}(G)|)+3|V_{3}(G)|=2n+|V_{3}(G)|\le 2n+\lfloor\frac{n}{2}\rfloor\le \frac{5n}{2}$. Thus, $|E(G)|\le \lfloor\frac{5n}{4}\rfloor$.
\end{proof}

\section{Proof of Theorem~\ref{jg1}}
Let $h(n)=\lfloor\frac{3(n-1)}{2}\rfloor$. We can easily conclude that
\begin{equation}
{\rm{\label{zhi}}}
h(n)={\left\{\begin{array}{ll}
  \frac{3n}{2}-2,& \text {if}~n~\text{is~even}, \\
  \frac{3(n-1)}{2},& \text {if}~n~\text{is~odd}.\\
\end{array}\right.}
\end{equation}
We are going to prove that $ex_{_\mathcal{OP}}^{c}(n,S_{2,2})=h(n)$ for all $n\ge 6$. To show that $ex_{_\mathcal{OP}}^{c}(n,S_{2,2})\ge h(n)$, we merely need to construct a connected $S_{2,2}$-free outerplanar graph with $n$ vertices and $h(n)$ edges. Let $G_{n}=K_{1}+(\lfloor \frac{n-1}{2}\rfloor K_{2} \cup \varepsilon K_{1})$, where $\varepsilon \equiv n-1~(\textup{mod}~2)$~(see Figure~\ref{figuregn}). Obviously, $G_{n}$ is a connected outerplanar graph. By the construction of $G_{n}$ and equation (\ref{zhi}), $|V(G_{n})|=1+2\lfloor \frac{n-1}{2}\rfloor +\varepsilon=n$, $|E(G_{n})|=\lfloor\frac{n-1}{2}\rfloor+(n-1)=\lfloor\frac{3(n-1)}{2}\rfloor=h(n)$. Notice that $G_{n}$ contains only one vertex of degree at least $3$. Thus, $G_{n}$ is $S_{2,2}$-free. Therefore, $G_{n}$ is the desired graph, and $ex_{_\mathcal{OP}}^{c}(n,S_{2,2})\ge h(n)$.
\begin{figure}[htbp]
    \centering
    \label{figuregn}
    \subfigure[~When $n=12$.]{
        \includegraphics[width=4.2cm]{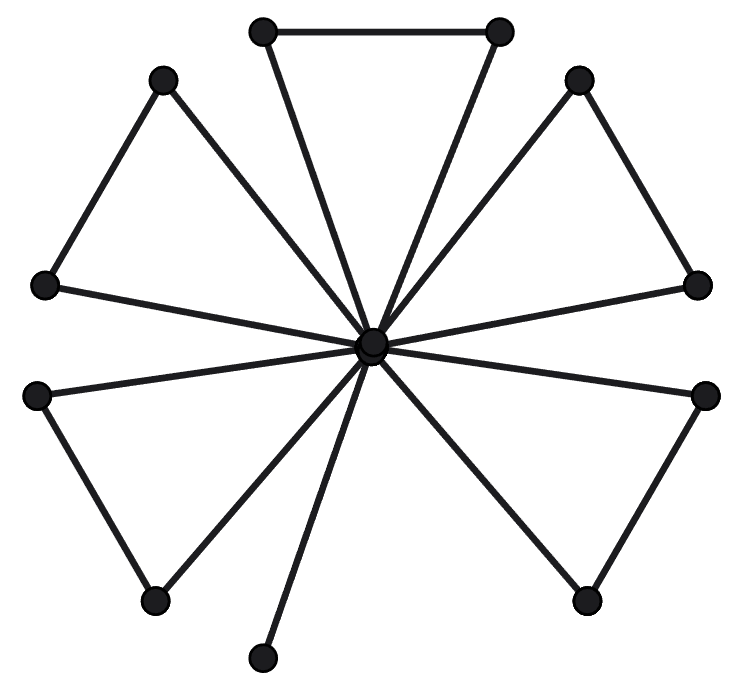}
    }
    \hspace{0.05\textwidth} 
    \subfigure[~When $n=13$.]{
        \includegraphics[width=4.2cm]{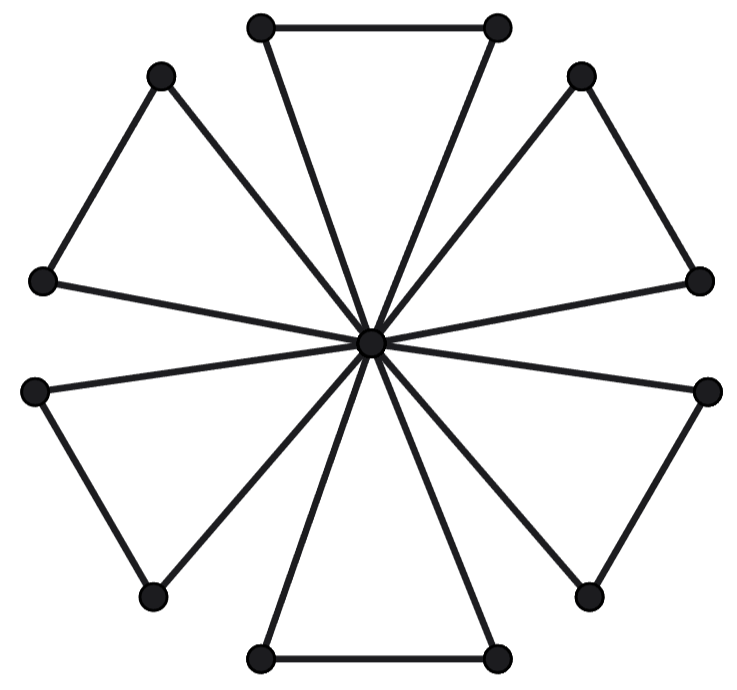}
    }
    \caption{~The graph $G_{n}$ when $n\in \{12,13\}$.}
\end{figure}

It remains to prove that $ex_{_\mathcal{OP}}^{c}(n,S_{2,2})\le h(n)$. Let $G$ be a connected $S_{2,2}$-free outerplanar graph with $n$ vertices and $ex_{_\mathcal{OP}}^{c}(n,S_{2,2})$ edges. Once $|E(G)|\le h(n)$ is proven, the proof of Theorem~\ref{jg1} can be completed. We prove $|E(G)|\le h(n)$ for all $n\ge 6$ by induction on $n$.

If $G$ is $2$-connected, then by Lemma~\ref{lemjg1}, $|E(G)|\le \lfloor\frac{5n}{4}\rfloor$. We have $\frac{3(n-1)}{2}-\frac{5n}{4}=\frac{n-6}{4}\ge 0$ as $n\ge 6$. It follows that
$h(n)- \lfloor\frac{5n}{4}\rfloor=\lfloor\frac{3(n-1)}{2}\rfloor-\lfloor\frac{5n}{4}\rfloor \ge 0$. Thus, $|E(G)|\le \lfloor\frac{5n}{4}\rfloor \le h(n)$, as desired.

Next, we assume that $G$ is not $2$-connected. Let $G_{1},G_{2},\ldots,G_{r}$ be the blocks of $G$, where $r\ge 2$. For each $i\in [r]$, let $n_{i}=|V(G_{i})|$ and $e_{i}=|E(G_{i})|$. Thus,
\begin{equation}
{\rm{\label{ke}}}
|E(G)|=\sum_{i=1}^{r} e_{i}.
\end{equation}
and by Lemma~\ref{lem3} we have
\begin{equation}
{\rm{\label{kv}}}
n=\sum_{i=1}^{r} n_{i} - r + 1 .
\end{equation}
We first present several observations that will be utilized in the subsequent proof.
\begin{ob}
{\rm{\label{fr}}}
$|E(G)|\le 2n-r-2$.
\end{ob}

Notice that $2\le n_{i}\le n-1$ for each $i\in [r]$. By Lemma~\ref{lem2}, $e_{i}\le 2n_{i}-3$. By equalities (\ref{ke}) and (\ref{kv}), $|E(G)|=\sum_{i=1}^{r} e_{i}\le \sum_{i=1}^{r}(2n_{i}-3)=2\sum_{i=1}^{r} n_{i}-3r=2(n+r-1)-3r=2n-r-2$.

\begin{ob}
{\rm{\label{f45}}}
For some $\ell\in [r]$, if $n_{\ell}=5$, or $n_{\ell}=4$ and $G_{\ell}$ is not adjacent to $M_{2}$, then $e_{\ell}\le 2n_{\ell}-4$.
\end{ob}

By Lemma~\ref{m5} and Lemma~\ref{m4}, $G_{\ell}\not\cong M_{n_{\ell}}$. Thus, $e_{\ell}< |E(M_{n_{\ell}})|=2n_{\ell}-3$ by Lemma~\ref{lem2}.

\begin{ob}
{\rm{\label{f6}}}
For some $\ell\in [r]$, if $n_{\ell}\ge 6$, then $e_{\ell}\le \lfloor\frac{5n_{\ell}}{4}\rfloor$.
\end{ob}

The block $G_{\ell}$ is a $2$-connected $S_{2,2}$-free outerplanar graph with $n_{\ell}$ vertices. By Lemma~\ref{lemjg1}, $e_{\ell}\le \lfloor\frac{5n_{\ell}}{4}\rfloor$.

We proceed with the proof by distinguishing cases according to $n$.

\noindent {\bf Case~1}.~$n\in \{6,7\}$.

For $n\in \{6,7\}$, we have $h(n)=\lfloor\frac{3(n-1)}{2}\rfloor=2n-5$. If $r\ge 3$, then $|E(G)|\le 2n-r-2\le 2n-5=h(n)$ by Observation~\ref{fr}, as desired. Next, we assume that $r=2$ and $n_{1}\le n_{2}$, without loss of generality. So $2\le n_{1}\le \lfloor\frac{n}{2}\rfloor=3$. By equality (\ref{kv}), $n_{2}=n-n_{1}+1$. It follows that $4\le n_{2}\le 6$. Notice that if $n_{2}=4$, then $n_{1}=n-n_{2}+1\neq 2$. If $4\le n_{2}\le 5$, then $e_{2}\le 2n_{2}-4$ by Observation~\ref{f45}. Thus, by equality (\ref{ke}), $|E(G)|=\sum_{i=1}^{r} e_{i}\le (2n_{1}-3)+(2n_{2}-4)=2(n+1)-7=2n-5=h(n)$, as desired. If $n_{2}=6$, then by equality (\ref{kv}) and $r=2$ we have $n_{1}=2$ and $n=7$. By Observation~\ref{f6}, $e_{2}\le \lfloor\frac{5n_{2}}{4}\rfloor$ and so $|E(G)|=\sum_{i=1}^{r} e_{i}\le (2n_{1}-3)+\lfloor\frac{5n_{2}}{4}\rfloor= 1+7=8<9=h(7)$, as desired.
\bigskip

\noindent {\bf Case~2}.~$n\ge 8$.

We begin by stating two claims that will be used in the proof of Case~2.

\begin{cla}
{\rm{\label{cla1}}}
For all $n\ge 8$, if there exists a vertex of degree 1 or two adjacent vertices of degree 2 in $G$, then $|E(G)|\le h(n)$.
\end{cla}
\textit{Proof}.~Let $w\in V(G)$ such that $d_{G}(w)=1$, and let $G'=G\setminus\{w\}$. We have $|E(G')|=|E(G)|-1~and~|V(G')|=n-1$. Since $|V(G')|=n-1\ge 7$ and $G'$ is a connected $S_{2,2}$-free outerplanar graph, by the induction hypothesis,
\begin{equation*}
|E(G')| \le h(|V(G')|)=h(n-1).
\end{equation*}
By equality (\ref{zhi}),
\begin{equation*}
h(n)-h(n-1)={\left\{\begin{array}{ll}
   \big(\frac{3n}{2}-2 \big)-\frac{3((n-1)-1)}{2}=1,& \text {if}~n~\text{is~even}, \\
  \frac{3(n-1)}{2}- \big(\frac{3(n-1)}{2}-2 \big)=2,& \text {if}~n~\text{is~odd}.\\
\end{array}\right.}
\end{equation*}
Combining this with $|E(G)|=|E(G')|+1\le h(n-1)+1$, we get
\begin{equation*}
h(n)-|E(G)|\ge h(n)-(h(n-1)+1)\ge 0.
\end{equation*}
It follows that $|E(G)|\le h(n)$, as desired.

Let $\{u,v\}\subseteq V(G)$ such that $d_{G}(u)=d_{G}(v)=2$, $uv\in E(G)$, and let $G^{*}=G\setminus\{u,v\}$. We have $|E(G^{*})|=|E(G)|-3$~and~$|V(G^{*})|=n-2$. Since $|V(G^{*})|=n-2\ge 6$ and $G^{*}$ is a connected $S_{2,2}$-free outerplanar graph, by the induction hypothesis,
\begin{equation*}
|E(G^{*})| \le h(|V(G^{*})|)=h(n-2).
\end{equation*}
By equality (\ref{zhi}),
\begin{equation*}
{\rm{\label{shan2}}}
h(n)-h(n-2)={\left\{\begin{array}{ll}
  \big(\frac{3n}{2}-2 \big)- \big(\frac{3(n-2)}{2}-2 \big)=3,& \text {if}~n~\text{is~even}, \\
  \frac{3(n-1)}{2}-\frac{3((n-2)-1)}{2}=3,& \text {if}~n~\text{is~odd}.\\
\end{array}\right.}
\end{equation*}
Combining this with $|E(G)|=|E(G^{*})|+3\le h(n-2)+3$, we get
\begin{equation*}
h(n)-|E(G)|\ge h(n)-(h(n-2)+3)\ge 0.
\end{equation*}
It follows that $|E(G)|\le h(n)$, as desired.

This completes the proof of Claim~\ref{cla1}.
$\hfill \square$

\begin{cla}
{\rm{\label{cladk}}}
For all $n\ge 8$, if there exists an endblock of order at most 5 in $G$, then $|E(G)|\le h(n)$.
\end{cla}
\textit{Proof}.~For some $\ell\in [r]$, let $G_{\ell}$ be an endblock of $G$ and $2\le n_{\ell}\le 5$. If $n_{\ell}\in \{2,3\}$, then obviously $G_{\ell}\cong M_{n_{\ell}}$. Notice that there exists a vertex of degree 1 in $G_{\ell}$ when $n_{\ell}=2$ and there exist two adjacent vertices of degree 2 in $G_{\ell}$ when $n_{\ell}=3$, by Claim~\ref{cla1}, we have $|E(G)|\le h(n)$. Next, we assume that all endblocks of $G$ have an order of at least 4.

If $n_{\ell}=5$, then by Lemma~\ref{m5} we have $G_{\ell}\not\cong M_{5}$. By Lemma~\ref{lem1}, $G_{\ell}$ has a Hamilton cycle. So $G_{\ell}$ contains two adjacent vertices of degree 2 and so Claim~\ref{cladk} holds by Claim~\ref{cla1}.

Next, we assume that $n_{\ell}=4$. Without loss of generality, let $\ell=1$ and $G_{i+1}$ be adjacent to $G_{i}$ for all $i\in [2]$. Since $G_{1}$ is a block of order 4, $G_{1}\cong C_{4}$ or $G_{1}\cong M_{4}$. If $G_{1}\cong C_{4}$, then there exist two adjacent vertices of degree 2 in $G_{1}$ and so $|E(G)|\le h(n)$ by Claim~\ref{cla1}. We next assume that $G_{1}\cong M_{4}$. By Lemma~\ref{lem2}, $e_{1}=2n_{1}-3=5$. By Lemma~\ref{m4}, $G_{2}\cong M_{2}$ and $G_{3}\cong M_{2}$. So $n_{2}=n_{3}=2$ and $e_{2}=e_{3}=1$. Let $\overline{G}=G\big[\bigcup_{i=1}^{3} E(G_{i})\big]$ and $\widetilde{G}=G[E(G)\setminus E(\overline{G})]$. To clearly identify the names of each vertex, we have depicted the structure of graph $G$ in Figure~\ref{figurem4jg}, where the part enclosed by a dashed circle is $\widetilde{G}$ and the remaining part is $\overline{G}$. Denote $|V(\widetilde{G})|$~($|V(\overline{G})|$) and $|E(\widetilde{G})|$~($|E(\overline{G})|$) by $\tilde{n} $~($\bar{n}$) and $\tilde{e}$~($\bar{e}$), respectively. By the construction of $\overline{G}$, $\bar{n}=6$ and $\bar{e}=\sum_{i=1}^{3}e_{i}=5+1+1=7$. Thus,
\begin{equation*}
{\rm{\label{m4shi}}}
\tilde{n}=n-\bar{n}+1=n-5~\textup{and}~|E(G)|=\tilde{e}+\bar{e}=\tilde{e}+7.
\end{equation*}
Notice that $G$ has at least two endblocks as $G$ is outerplanar. We have $\overline{G}$ must contain an endblock. Combining this with all endblocks of $G$ have an order of at least 4, $\tilde{n}\ge 4$. We proceed the prove by considering cases based on the value of $\tilde{n}$.
\begin{figure}[htp]
	\centering
	\includegraphics[width=9.5cm]{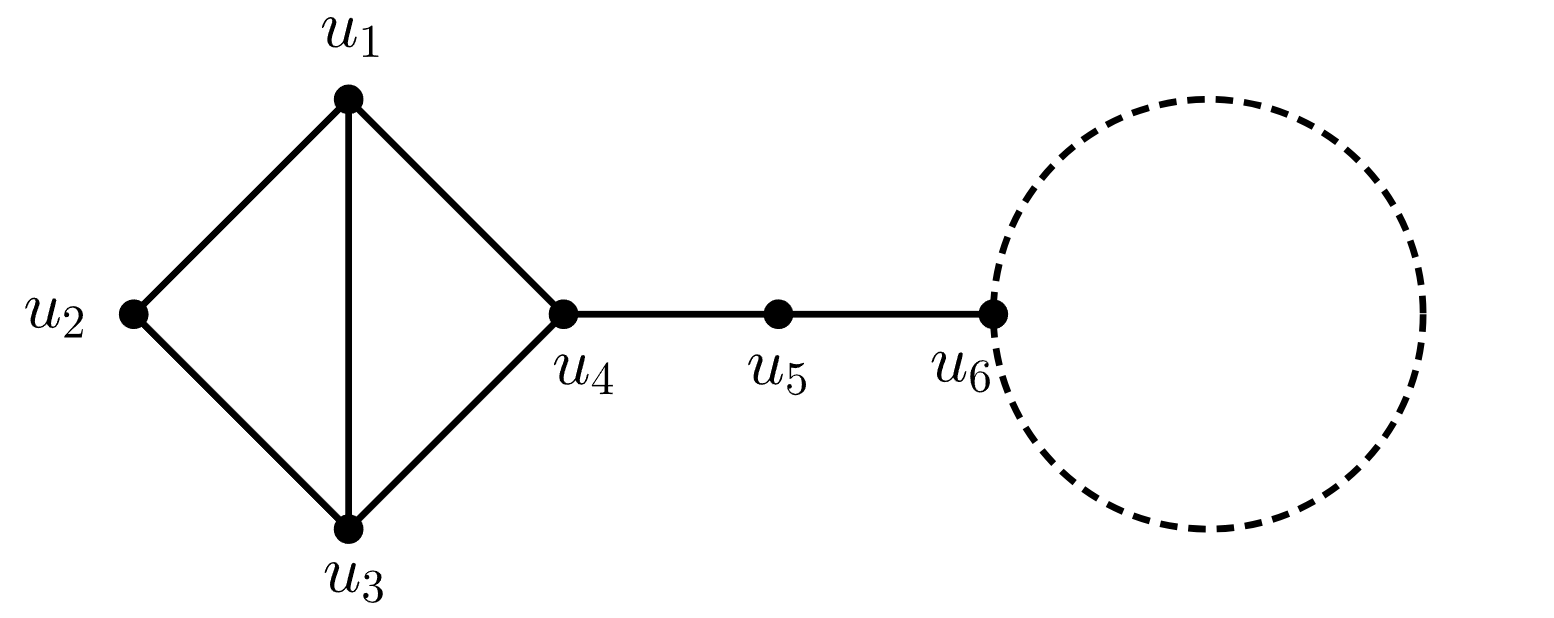}
	\caption{ The graph $G$ when $n_{1}=4$.}
	\label{figurem4jg}
\end{figure}

When~$\tilde{n}\in \{4,5\}$, $h(n)=h(\tilde{n}+5)=\lfloor\frac{3(\tilde{n}+5-1)}{2}\rfloor=\tilde{n}+8$ by $n=\tilde{n}+5$. If $\tilde{n}=4$, then by Lemma~\ref{lem2} we have $\tilde{e}\le 2\tilde{n}-3$ and so $|E(G)|=\tilde{e}+7\le 2\tilde{n}+4=\tilde{n}+8=h(n)$, as desired. If $\tilde{n}=5$, then by Observation~\ref{f45} we have $\tilde{e}\le 2\tilde{n}-4$ and so $|E(G)|=\tilde{e}+7\le 2\tilde{n}+3=\tilde{n}+8=h(n)$, as desired.

Next, we assume that~$\tilde{n}\ge 6$. Since $6\le \tilde{n}< n$ and $\widetilde{G}$ is a connected $S_{2,2}$-free outerplanar graph, by the induction hypothesis, $\tilde{e} \le h(\tilde{n})$. So $|E(G)|=\tilde{e}+7\le h(\tilde{n})+7=\lfloor\frac{3(\tilde{n}-1)}{2}\rfloor+7\le \frac{3(\tilde{n}-1)}{2}+7$. By equality (\ref{zhi}), $\frac{3n}{2}-2\le h(n)\le \frac{3(n-1)}{2}$. Thus, $h(n)-|E(G)|\ge \big(\frac{3n}{2}-2\big)-\big(\frac{3(\tilde{n}-1)}{2}+7\big)=\frac{3(n-\tilde{n})-15}{2}=\frac{3\times 5-15}{2}=0$ as $n=\tilde{n}+5$. It follows that $|E(G)|\le h(n)$, as desired.

This completes the proof of Claim~\ref{cladk}.
$\hfill \square$
\bigskip

If there exists an endblock of order at most 5 in $G$, then by Claim~\ref{cladk} we have $|E(G)|\le h(n)$, as desired. Next, we assume that every endblock in $G$ is of order at least 6. We may assume that $G_{1}$ is an endblock of $G$ and $H=G[E(G)\setminus E(G_{1})]$. For notational convenience, we denote $|E(H)|$ and $|V(H)|$ by $e_{0}$ and $n_{0}$, respectively. By the construction of $H$, we get $H$ has $r-1$ blocks, $|V(G)|=n_{0}+n_{1}-1$ and $|E(G)|=e_{0}+e_{1}$. Since $G$ is a non-2-connected outerplanar graph, $G$ has at least two endblocks. Notice that every endblock in $G$ is of order at least 6, so $n_{0}\ge 6$. Combining this with the fact that $H$ is a connected $S_{2,2}$-free outerplanar graph with $n_{0}$ vertices, by the induction hypothesis, $e_{0}\le h(n_{0})$. Thus, $|E(G)|=e_{0}+e_{1}\le h(n_{0})+e_{1}$. Since $n_{1}\ge 6$, $e_{1}\le \lfloor\frac{5n_{1}}{4}\rfloor\le \frac{5n_{1}}{4}$ by Lemma~\ref{lemjg1}. Hence, by equality (\ref{zhi}) and $n=n_{0}+n_{1}-1$,
\begin{equation*}
\begin{aligned}
h(n)-|E(G)|
& \ge h(n)-(h(n_{0})+e_{1})\\
& \ge \bigg(\frac{3n}{2}-2\bigg)-\bigg(\frac{3(n_{0}-1)}{2}+\frac{5n_{1}}{4}\bigg)\\
& =\frac{3(n_{1}-1)-1}{2}-\frac{5n_{1}}{4}=\frac{n_{1}-8}{4}.
\end{aligned}
\end{equation*}
If $n_{1}\ge 8$, then $h(n)-|E(G)|\ge \frac{n_{1}-8}{4}\ge 0$, as desired. It only remains to prove that $|E(G)|\le h(n)$ holds when $n_{1}\in\{6,7\}$. By Observation~\ref{f6}, $e_{1}\le \lfloor\frac{5n_{1}}{4}\rfloor=n_{1}+1$. Since $G_{1}$ is a block of order $n_{1}$, $G_{1}$ is 2-connected. By Lemma~\ref{lem1}, $G_{1}$ has a Hamilton cycle which contains $n_{1}$ edges. Thus, the Hamilton cycle of $G_{1}$ has at most one chord. Therefore, $G_{1}$ must contain two adjacent vertices of degree 2 and so Theorem~\ref{jg1} holds by Claim~\ref{cla1}.

In conclusion, we complete the proof of Theorem~\ref{jg1}.
$\hfill \square$

\section{Proof of Theorem~\ref{jg2}}

Let $h(n)=\lfloor\frac{3(n-1)}{2}\rfloor$ and $$h_{1}(n)={\left\{\begin{array}{ll}
  h(n),& if~n\ge 6~and~n\neq 10; \\
  14,& if~n=10.\\
\end{array}\right.}$$
We only need to prove that $ex_{_\mathcal{OP}}(n,S_{2,2})=h_{1}(n)$ for all $n\ge 6$. We first prove that $ex_{_\mathcal{OP}}(n,S_{2,2})\ge h_{1}(n)$. When $n=10$, let $G_{0}=2M_{5}$. By Lemma~\ref{lem2}, $|E(G_{0})|= 2(2\times 5-3)=14$. Obviously, $G_{0}$ is an outerplanar graph with 10 vertices and $G_{0}$ is $S_{2,2}$-free as $|V(M_{5})|=5<6=|V(S_{2,2})|$. Hence, $ex_{_\mathcal{OP}}(n,S_{2,2})\ge |E(G_{0})|=14$, as desired. When $n\ge 6$ and $n\neq 10$, we have $ex_{_\mathcal{OP}}^{c}(n,S_{2,2})=h(n)$ as Theorem~\ref{jg1} holds. So from the definitions of $ex_{_\mathcal{OP}}(n,S_{2,2})$ and $ex_{_\mathcal{OP}}^{c}(n,S_{2,2})$, $ex_{_\mathcal{OP}}(n,S_{2,2})\ge ex_{_\mathcal{OP}}^{c}(n,S_{2,2})=h(n)$, as desired. Thus, $ex_{_\mathcal{OP}}(n,S_{2,2})\ge h_{1}(n)$ for all $n\ge 6$.

Subsequently, we shall prove $ex_{_\mathcal{OP}}(n,S_{2,2})\le h_{1}(n)$ for all $n\ge 6$. Let $G$ be an $S_{2,2}$-free outerplanar graph with $n$ vertices and $ex_{_\mathcal{OP}}(n,S_{2,2})$ edges. We are going to show $|E(G)|\le h_{1}(n)$. Notice that $\lfloor\frac{3(n-1)}{2}\rfloor=13< 14$ when $n=10$. If $G$ is connected, then by Theorem~\ref{jg1} we have $|E(G)|\le h(n)\le h_{1}(n)$ for all $n\ge 6$, as desired. So we next assume that $G$ is disconnected. Let $H_{1},H_{2},\ldots,H_{s}$ be the components of $G$, where $s\ge 2$. For each $i\in [s]$, let $n_{i}=|V(H_{i})|$ and $e_{i}=|E(H_{i})|$. So
\begin{equation}
{\rm{\label{ev}}}
|E(G)|=\sum_{i=1}^{s}e_{i}~\textup{and}~n=\sum_{i=1}^{s}n_{i}.
\end{equation}
For each $j \in [5]$, let $s_{j}$ denote the number of components of order $j$, and $s_{6}$ denote the number of components with order at least 6. So
\begin{equation}
{\rm{\label{ks}}}
s=\sum_{i=1}^{6}s_{i}.
\end{equation}
Without loss of generality, assume $H_{1},H_{2},\ldots,H_{s_{6}}$ are the components with order at least 6. By equalities (\ref{ev}) and (\ref{ks}),
\begin{equation}
{\rm{\label{dv}}}
n=\sum_{i=1}^{s}n_{i}=\sum_{i=1}^{s_{6}}n_{i}+\sum_{j=1}^{5}js_{j}.
\end{equation}

We complete the proof by dividing it into two cases based on the value of $n$.

\noindent {\bf Case~1}.~$n= 10$.

We are going to prove that $|E(G)|\le 14$. By equality (\ref{dv}) and $n_{i}\ge 6$ for each $i\in [s_{6}]$, $n\ge \sum_{i=1}^{s_{6}}n_{i}\ge 6s_{6}$. So $s_{6}\le 1$, otherwise contradicting $n= 10$. If $s_{6}= 1$, then $n_{1}\ge 6$ and by Theorem~\ref{jg1} we have $e_{1}\le h(n_{1})$. Notice that $1\le n_{i}\le 5$ for each $i\in \{2,3,\cdots,s\}$. So by equality (\ref{ev}) and Lemma~\ref{lem2}, $|E(G)|=\sum_{i=1}^{s}e_{i}=e_{1}+\sum_{i=2}^{s}e_{i}\le h(n_{1})+\sum_{i=2}^{s}(2n_{i}-3)=\lfloor\frac{3(n_{1}-1)}{2}\rfloor+(2(10-n_{1})-3)\le \frac{3(n_{1}-1)}{2}+(2(10-n_{1})-3)=14+\frac{3-n_{1}}{2}<14$, as desired. If $s_{6}= 0$, then $1\le n_{i}\le 5$ for each $i \in [s]$ and by Lemma~\ref{lem2} we have $e_{i}\le 2n_{i}-3$. Thus, by equality (\ref{ev}) and $s\ge 2$, $|E(G)|=\sum_{i=1}^{s}e_{i}\le \sum_{i=1}^{s}(2n_{i}-3)=2n-3s\le 2n-3\times 2=14$, as desired.
\bigskip

\noindent {\bf Case~2}.~$n\ge 6$ and $n\neq 10$.

We are going to prove that $|E(G)|\le h(n)$. By Theorem~\ref{jg1} and equality (\ref{dv}),
\begin{equation}
{\rm{\label{shi1}}}
\sum_{i=1}^{s_{6}}e_{i}\le \sum_{i=1}^{s_{6}}h(n_{i})= \sum_{i=1}^{s_{6}}\bigg\lfloor\frac{3(n_{i}-1)}{2}\bigg\rfloor\le \Bigg\lfloor\frac{3\sum_{i=1}^{s_{6}}n_{i}-3s_{6}}{2}\Bigg\rfloor=\Bigg\lfloor\frac{3\big(n-\sum_{j=1}^{5}js_{j}\big)-3s_{6}}{2}\Bigg\rfloor.
\end{equation}
By Lemma~\ref{lem2}, $e_{i}\le \textup{max}\{0,2n_{i}-3\}$. Combining this with equality (\ref{ks}),
\begin{equation}
{\rm{\label{shi2}}}
\begin{aligned}
\sum_{i=s_{6}+1}^{s}e_{i}\le \sum_{i=s_{6}+1}^{s}\textup{max}\{0,2n_{i}-3\}=\sum_{j=1}^{5}s_{j}\times \textup{max}\{0,2j-3\}=s_{2}+3s_{3}+5s_{4}+7s_{5}.
\end{aligned}
\end{equation}
Combining this with equality (\ref{ev}) and inequalities (\ref{shi1}), (\ref{shi2}), we obtain
\begin{equation*}
{\rm{\label{disedge}}}
\begin{aligned}
|E(G)|=\overset{s}{\underset{i=1}{\sum}}e_{i}
& = \sum_{i=1}^{s_{6}}e_{i}+\sum_{j=s_{6}+1}^{s}e_{i}\\
& \le \bigg\lfloor\frac{3(n-\sum_{j=1}^{5}js_{j})-3s_{6}}{2}\bigg\rfloor+(s_{2}+3s_{3}+5s_{4}+7s_{5})\\
& = \bigg\lfloor\frac{3(n-1)+3(1-\sum_{j=1}^{5}js_{j})-3s_{6}+2s_{2}+6s_{3}+10s_{4}+14s_{5}}{2}\bigg\rfloor\\
& = \bigg\lfloor\frac{3(n-1)}{2}+\frac{3(1-s_{6}-s_{1}-s_{2}-s_{3})-s_{2}-2s_{4}-s_{5}}{2}\bigg\rfloor.
\end{aligned}
\end{equation*}

Let $p=3(1-s_{6}-s_{1}-s_{2}-s_{3})-s_{2}-2s_{4}-s_{5}$. To prove $|E(G)|\le h(n)=\lfloor\frac{3(n-1)}{2}\rfloor$, it remains to prove that $p\le 0$. Since $s_{i}\ge 0$ for all $i\in [6]$, if $s_{6}+s_{1}+s_{2}+s_{3}\ge 1$, then $p\le 0$, as desired. Next, we suppose that $s_{6}+s_{1}+s_{2}+s_{3}=0$, this implies $s_{6}=s_{1}=s_{2}=s_{3}=0$ and $p=(3-s_{4}-s_{5})-s_{4}$. If $s_{4}+s_{5}\ge 3$, then $p\le 0$, as desired. We now consider $s_{4}+s_{5}\le 2$. By equality (\ref{ks}) and $s\ge 2$, we have $s=s_{4}+s_{5}\ge 2$. Thus, $s_{4}+s_{5}=2$ and so $p=1-s_{4}$. If $s_{4}\ge 1$, then $p\le 0$, as desired. So the only remaining case is $s_{4}=0$ and $s_{5}=2$. By equality (\ref{dv}), $n=5s_{5}=10$, contradicting $n\neq 10$. Therefore, $p\le 0$ holds in all cases.

In summary, the proof of Theorem~\ref{jg2} is completed.
$\hfill \square$

\section{Proof of Theorem~\ref{jg3}}

By Lemma~\ref{lem12}, $ex_{_\mathcal{OP}}(n,S_{p,q})\le 2n-3$. It remains to prove $ex_{_\mathcal{OP}}(n,S_{p,q})\ge 2n-3$. For $p\ge 3$ or $q\ge 4$, we only need to construct an $S_{p,q}$-free outerplanar graph with $n$ vertices and $2n-3$ edges. We consider the following two cases depending on the values of $p$ and $q$.

\noindent {\bf Case~1}.~$p= q=3$.

Let $T_{n}=K_{1}+P_{n-1}$~(see Figure $\ref{figureon}$). By the construction of $T_{n}$, $T_{n}$ is an outerplanar graph, $|V(T_{n})|=1+(n-1)=n$, and $|E(T_{n})|=(n-1)+(n-2)=2n-3$. Since there is only one vertex of degree at least 4 in $T_{n}$, and $S_{p,q}$ requires more than one vertex of degree at least 4, we conclude that $T_{n}$ is $S_{p,q}$-free. Thus, $ex_{_\mathcal{OP}}(n,S_{p,q})\ge |E(T_{n})|= 2n-3$, as desired.
\bigskip

\noindent {\bf Case~2}.~$p\ge 4$ or $q\ge 4$.

Let $O_{n}$ denote the maximal outerplanar graph with maximum degree $4$. When $n = 14$, the graph $O_{n}$ is depicted in Figure $\ref{figuremn}$. For $p\ge 4$ or $q\ge 4$, it is easy to see that $O_{n}$ is $S_{p,q}$-free because the maximum degree of $O_{n}$ is four. Thus, $ex_{_\mathcal{OP}}(n,S_{p,q})\ge |E(O_{n})|= 2n-3$, as desired.
\begin{figure}[htbp]
    \centering
    \subfigure[~The construction of $T_{n}$.]{
        \includegraphics[width=4.2cm]{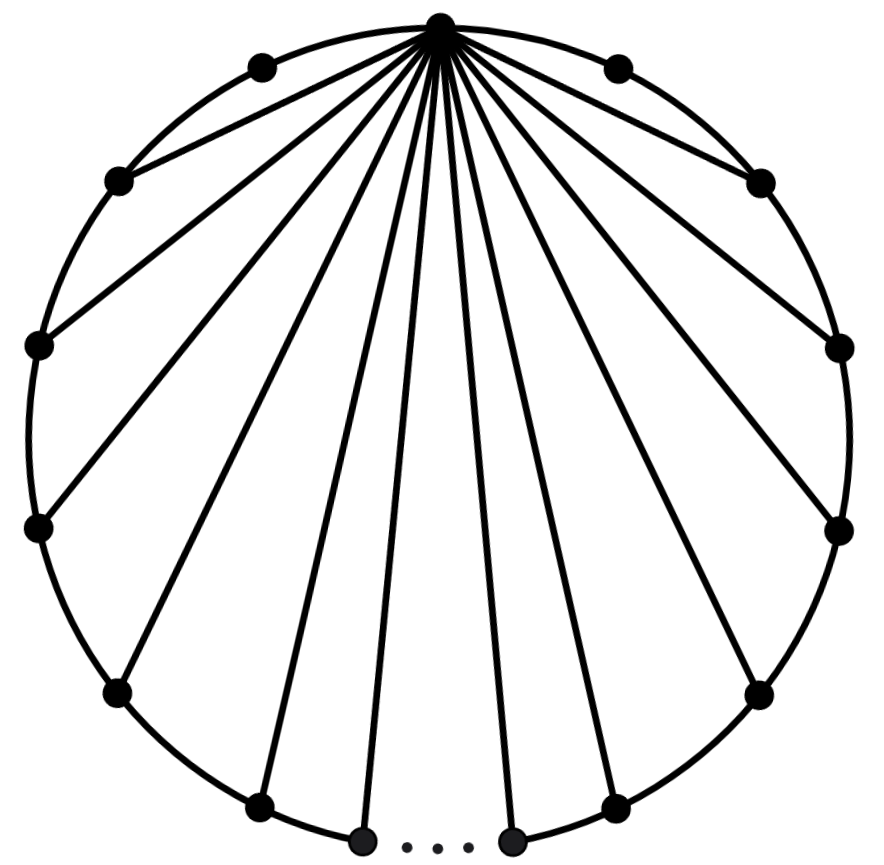}
        \label{figureon}
    }
    \hspace{0.05\textwidth} 
    \subfigure[~The construction of $O_{14}$.]{
        \includegraphics[width=4.2cm]{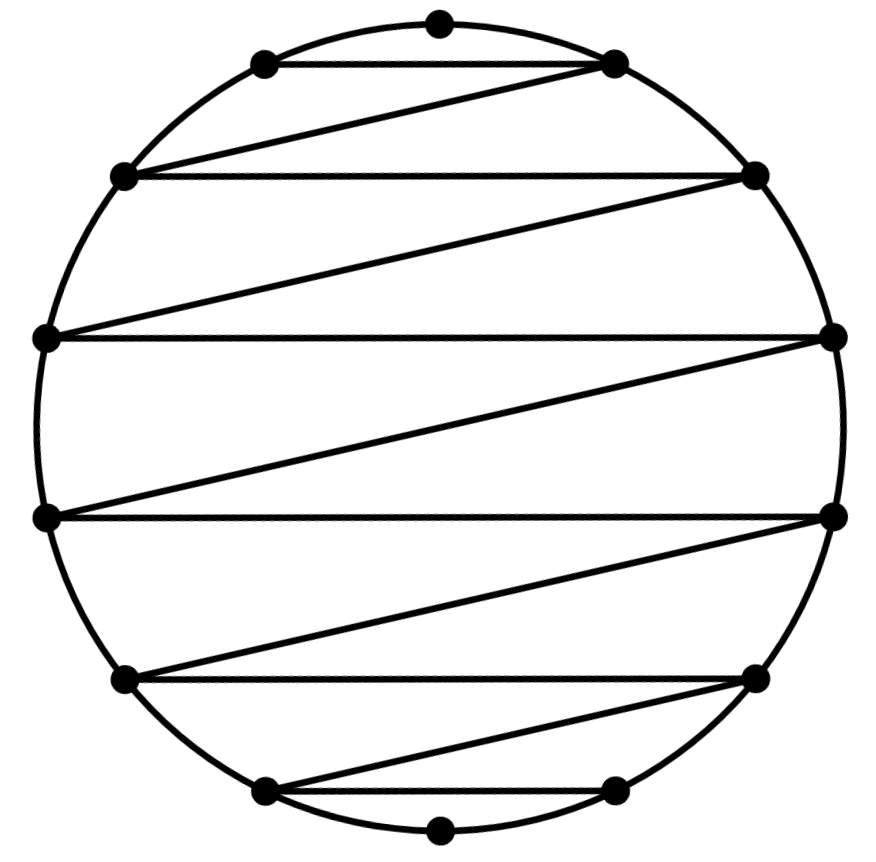}
        \label{figuremn}
    }
    \caption{~The graphs $T_{n}$ and $O_{14}$.}
\end{figure}

Therefore, Theorem~\ref{jg3} holds.
$\hfill \square$

\section{Proof of Theorem~\ref{jg4}}

Let $t$ and $i$ be integers such that $n=6t+i$, where $i\in \{0\}\cup [5]$, and let
\begin{equation}
{\rm{\label{fn}}}
f(n)={\left\{\begin{array}{ll}
  \frac{5n-3}{3},& if~n\equiv 0~(\textup{mod}~6); \\
  \frac{5n-5}{3},& if~n\equiv 1~(\textup{mod}~6); \\
  \frac{5n+i-9}{3},& if~n\equiv i~(\textup{mod}~6),~where~i\in\{2,3,4,5\}.
\end{array}\right.}
\end{equation}
To illustrate $ex_{_\mathcal{OP}}^{c}(n,S_{2,3})\ge f(n)$, for each $i\in \{ 0 \}\cup [5]$, we need merely construct an $S_{2,3}$-free connected outerplanar graph $H_{i}'$ with $6t+i$ vertices and $f(6t+i)$ edges. Let $H$ be the unique maximal outerplanar graph on 6 vertices with maximum degree of 4~(see Figure~\ref{figurem61}). By Lemma~\ref{lem2}, $|E(H)|=2|V(H)|-3= 2\times 6-3=9$. Since $|V(H)|=6<7=|V(S_{2,3})|$, $H$ is $S_{2,3}$-free. We use $H$ as the base block and construct $H_{i}'$ for each $i\in \{ 0 \}\cup [5]$ below.
\begin{figure}[htp]
	\centering
	\includegraphics[width=3.4cm]{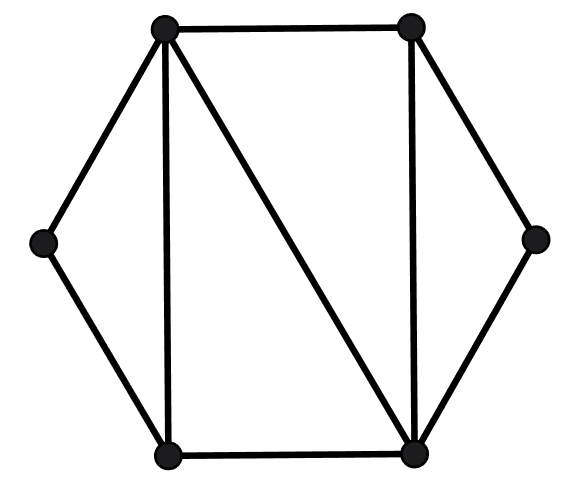}
	\caption{The graph $H$.}
	\label{figurem61}
\end{figure}

We first construct a connected outerplanar graph $H_{0}'$ as follows. By $n\ge 7$, we have $t\ge 1$. If $t=1$, then let $H_{0}'\cong H$. If $t\ge 2$, then first take $t$ disjoint copies of $H$, denoted $H_{1},H_{2},\ldots,H_{t}$. By the structure property of $H$, each $H_{i}$~(for $i\in [t]$) contains exactly two distinct vertices of degree 2, say $u_{i}$ and $v_{i}$. Let $H_{0}'$ be the graph obtained from $H_{1}\cup H_{2}\cup \ldots \cup H_{t}$ by adding $t-1$ edges: for each $i\in [t-1]$, add an edge between $v_{i}$ and $u_{i+1}$~(see Figure~\ref{figurehb}). We then construct a connected outerplanar graph $H_{i}'$ for each $i\in [5]$ as follows. Notice that $H_{0}'$ has exactly two vertices of degree 2, $u_{1}$ and $v_{t}$. For each $i\in[5]$, let $H_{i}'$ be the graph obtained from $H_{0}'$ by joining $v_{t}$ and one of the vertices of degree at most 2 in $M_{i}$, which is depicted in Figure~\ref{figurehbi} when $t=3$, where the part enclosed by a dashed circle is $M_{i}$.
\begin{figure}[htbp]
    \centering
    \subfigure[~The graph $H_{0}'$ when $t=3$.]{
        \includegraphics[width=11.5cm]{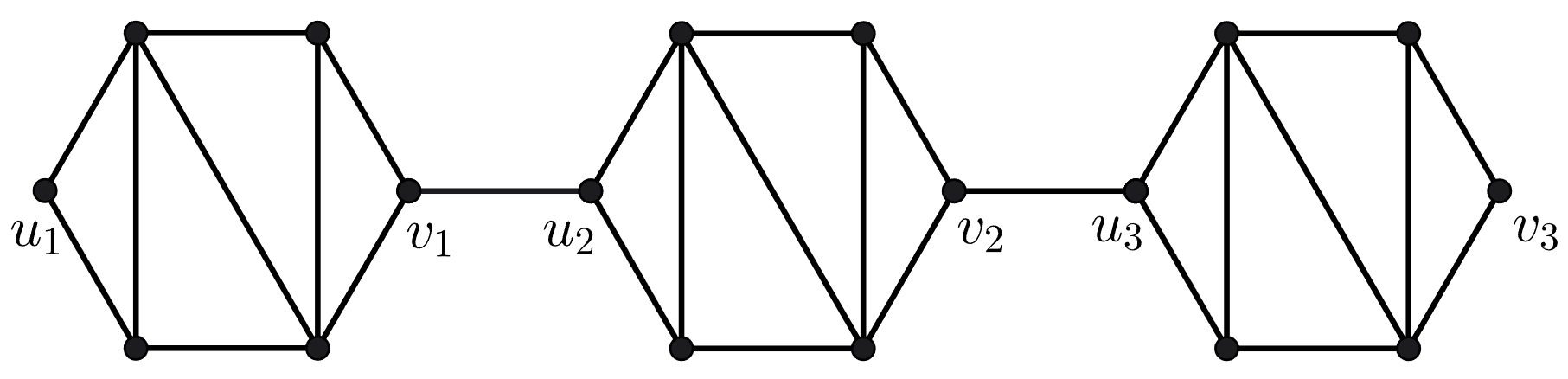}
        \label{figurehb}
    }
    \hspace{0.005\textwidth} 
    \subfigure[~The graph $H_{i}'$ when $t=3$, where $i\in \textup{[5]}$ and the part enclosed by a dashed circle is $M_{i}$.]{
        \includegraphics[width=15.6cm]{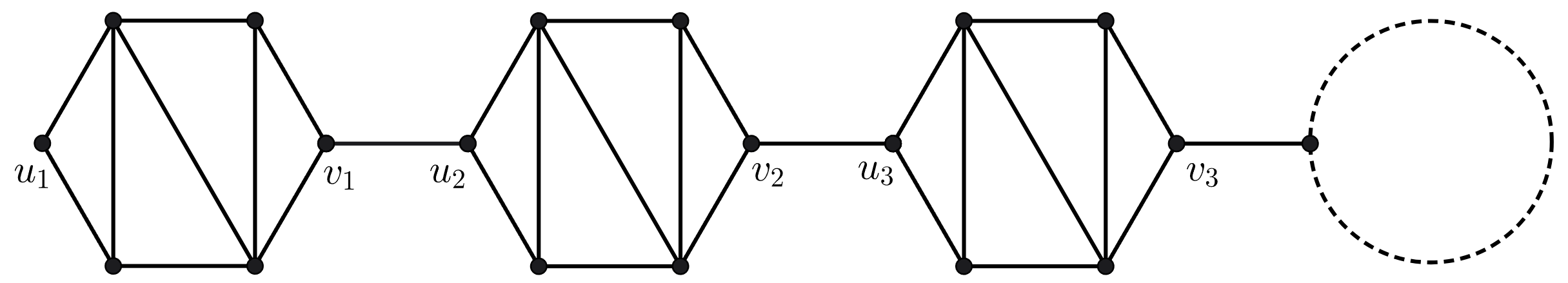}
        \label{figurehbi}
    }
    \caption{~The graph $H_{i}'$ when $t=3$, where $i\in \{0\}\cup \textup{[5]}$.}
\end{figure}

We are going to prove that $|V(H_{i}')|=6t+i$ and $|E(H_{i}')|=f(6t+i)$ for each $i\in \{0\}\cup [5]$. For $i= 0$, by the construction of $H_{0}'$, $|V(H)|=6$ and $|E(H)|=9$, we have $|V(H_{0}')|=t|V(H)|=6t$, and $|E(H_{0}')|=t|E(H)|+(t-1)=10t-1=\frac{5(6t)-3}{3}=f(6t)$, as desired.
For each $i\in [5]$, we have $|V(M_{i})|=i$ and $|E(M_{i})|=\textup{max}\{0,2i-3\}$ by Lemma~\ref{lem2}. Thus, by the construction of $H_{i}'$, $|V(H_{i}')|=|V(H_{0}')|+|V(M_{i})|=6t+i$, and $|E(H_{i}')|=|E(H_{0}')|+|E(M_{i})|=10t+\textup{max}\{0,2i-3\}$. Thus, $|E(H_{i}')|=10t+0=\frac{5(6t+i)-5}{3}=f(6t+i)$ if $i=1$, and $|E(H_{i}')|=10t+(2i-3)=\frac{5(6t+i)+i-9}{3}=f(6t+i)$ if $i\in\{2,3,4,5\}$, as desired.

To complete the proof of Theorem~\ref{jg4}, it suffices to verify that $H_{i}'$ is $S_{2,3}$-free for each $i\in \{ 0 \}\cup [5]$. To proceed with the proof, we first establish the following claim.
\begin{cla}
{\rm{\label{cla34e}}}
Let $G$ be an outerplanar graph. For any $(3,4)$-degree edge $xy$ in $E(G)$, if $N_{G}(x)\cap N_{G}(y)\neq \emptyset$, then $G$ is $S_{2,3}$-free.
\end{cla}
\textit{Proof}.~Suppose, by way of contradiction, that $G$ is an outerplanar graph satisfying the given condition but contains a subgraph isomorphic to $S_{2,3}(x,y;X,Y)$, so we have $|X|=2$, $|Y|=3$ and $X\cap Y= \emptyset$. Since $N_{G}(x)\cap N_{G}(y)\neq \emptyset$, let $z\in N_{G}(x)\cap N_{G}(y)$. By $X\cap Y= \emptyset$, we have $z\notin X\cap Y$. Thus, if $z\in X$, then $z\notin Y$, and so $\{z,x\}\cup Y\subseteq N_{G}(y)$, $d_{G}(y)\ge 2+|Y|=5$, contradicting $d_{G}(y)=4$. If $z\notin X$, then $\{z,y\}\cup X\subseteq N_{G}(x)$, and so $d_{G}(x)\ge 2+|X|=4$, contradicting $d_{G}(x)=3$. Therefore, $G$ is $S_{2,3}$-free.
$\hfill \square$
\bigskip

We now verify that $H_{i}'$ is $S_{2,3}$-free for each $i\in \{0\}\cup [5]$. Noting that when $t$ is fixed, $H_{i}' \subseteq H_{5}'$ for all $i \in \{0\} \cup [4]$, it suffices to show that $H_{5}'$ is $S_{2,3}$-free for all $t \ge 2$. Suppose, for contradiction, that $H_{5}'$ contains a subgraph $F$ isomorphic to $S_{2,3}$. Let $uv$ be the $(3,4)$-degree edge in $F$. For notational convenience, we denote the endblock of $H_{5}'$ that is isomorphic to $M_{5}$ as $H_{t+1}$. First, if both $u$ and $v$ lie in $H_{j}$ for some $j\in[t+1]$, then $N_{H_{5}'}(u) \cap N_{H_{5}'}(v) \neq \emptyset$ by the structural properties of $H_{j}$. This aligns with Claim~\ref{cla34e}, which implies $H_{5}'$ is $S_{2,3}$-free, a contradiction. Next, assume that $u$ lies in $H_{j}$ and $v$ lies in $H_{j+1}$ for some $j\in[t]$. A key property of $H_{5}'$, constructed from $H_{1},H_{2},\ldots,H_{t+1}$, is that all edges added during its construction connect only vertices of degree at most 2. Since $uv$ is an edge connecting $H_{j}$ and $H_{j+1}$, this implies $d_{H_{5}'}(v)\le 3$, contradicting the requirement that $d_{H_{5}'}(v)= 4$. Thus, our initial assumption is false, and $H_{5}'$ is $S_{2,3}$-free for all $t \ge 2$.

Hence, for each $i\in \{0\}\cup [5]$, $H_{i}'$ is an $S_{2,3}$-free connected outerplanar graph with $6t+i$ vertices and $f(6t+i)$ edges. Therefore, $ex_{_\mathcal{OP}}^{c}(n,S_{2,3})\ge |E(H_{i}')|=f(n)$. This complete the proof of Theorem~\ref{jg4}.
$\hfill \square$\\

\noindent {\bf Acknowledgments.} Yongxin Lan is partially supported by the National Natural Science Foundation of Hebei Province (No. A2025202034).


\end{sloppypar}

\end{document}